\documentclass[12pt,a4paper,twoside]{article}
\usepackage{booktabs}
\usepackage{multirow}
\usepackage{geometry}
\usepackage[UKenglish]{babel}
\usepackage[T1]{fontenc}
\usepackage{amsfonts,amsmath,amsthm,amssymb,mathrsfs}
\usepackage{fullpage}
\usepackage{xcolor,bm, bbm}
\usepackage{paralist}
\usepackage{dsfont}
\usepackage{float}
\usepackage{caption}
\usepackage{xfrac}
\usepackage[colorlinks, linkcolor=blue,citecolor=blue, anchorcolor=blue,urlcolor=blue,hypertexnames=false]{hyperref}
\usepackage{graphicx}
\usepackage{tikz,ifthen}
\usepackage{pgfplots}
\usepackage{indentfirst} 
\pgfplotsset{compat=1.18} 
\usepgfplotslibrary{fillbetween}
\usepackage{tikz}
\usetikzlibrary{arrows.meta,calc,positioning}

\newtheorem{theorem}{Theorem}[section]
\newtheorem{lemma}[theorem]{Lemma}

\theoremstyle{definition}

\numberwithin{equation}{section}

\DeclareMathOperator{\N}{\mathbb{N}}

\newcommand{\xn}{(x_n)_{n \in \mathbb{N}} }

\title{\bf Weak Inhomogeneous Poissonian Pair Correlation and Equidistribution}

\author{Zhiqin Tang \and Qing-Long Zhou}

\date{}

\begin{document}

\maketitle
\vspace{-1.5\baselineskip} 
\begin{abstract}

This paper investigates inhomogeneous Poissonian pair correlation (PPC), its weak form, and equidistribution. We establish that weak inhomogeneous PPC does not imply inhomogeneous PPC. Furthermore, we construct a sequence satisfying weak inhomogeneous PPC that fails to be equidistributed, which stands in sharp contrast to the homogeneous case. Finally, we prove that for distinct $\gamma_1, \gamma_2 \in (0, \frac{1}{2}]$, weak $\gamma_1$-PPC does not imply weak $\gamma_2$-PPC, showing that different inhomogeneous parameters give rise to mutually independent notions of weak inhomogeneous PPC.

\end{abstract}

\section{Introduction}
Understanding the distribution of sequences of real numbers modulo one is a well-studied and important problem in mathematics (see  \cite{B12,KN74}). We begin with a brief overview of three central concepts --- equidistribution, Poissonian pair correlation (PPC) and weak PPC --- along with selected results that motivate the present work.

\subsection{Equidistribution, PPC and Weak PPC}

A real sequence $\xn$ is said to be equidistributed (or uniformly distributed modulo one, abbreviated as u.d. mod 1) if for any interval $[a,b]\subseteq[0,1)$ we have
\begin{equation*}
  \lim_{N \to \infty} \frac{\sharp \{1\le n\le N\colon x_n \bmod 1 \in [a, b]\}}{N}=b-a.
 \end{equation*}
This concept, formally introduced by Weyl in \cite{W16}, has attracted extensive interest in recent decades owing to its close connections with number theory, fractal geometry and probability theory. Weyl criterion \cite[Theorem 3.2]{KN74} yields that any polynomial sequence $(x_n)$ --- where $x_n=\alpha_d n^d+\cdots+\alpha_1n+\alpha_0$ with at least one irrational coefficient --- is equidistributed. For more about  equidistribution, we refer to  \cite{B12,KN74}.

A more fine scale understanding of the distribution of a sequence is provided by the pair correlation statistics. Namely, given a sequence of real numbers $(x_n), \delta\in[0,1]$ and $s>0,$ we would like to understand the asymptotic behaviour of
\[\frac{\sharp \left\{ 1 \leq m \neq n \leq N\colon \|x_m - x_n\| \leq \frac{s}{N^{\delta}} \right\}}{N^{2-\delta}} \ \ \text{as} \ \ N\to\infty.\]
Here and throughout \(\| \cdot \|\) denotes the distance to the nearest integer.
For $\delta=0,$ the asymptotic behavior is equivalent to equidistribution \cite[Theorem 1]{HZ24}.

A sequence \((x_n)_{n \in \mathbb{N}}\) is said to have PPC ($\delta=1$) if for every fixed 
\(s > 0\),
\[
\lim_{N \to \infty}\frac{\sharp \left\{ 1 \leq m \neq n \leq N\colon \|x_m - x_n\| \leq \frac{s}{N} \right\}}{N} = 2s.
\]
 PPC behavior of a sequence is related to its equidistribution property. A sequence having PPC must be uniformly distributed \cite{ALP18,GL17,Ste20}, but the converse is not necessarily true.  The Kronecker  sequence $(n\alpha)_{n\in\N}$ is uniformly distributed for irrational $\alpha$, but fails to have PPC for any $\alpha\in\mathbb{R}$. Further examples can be found in \cite{HK19,PS19}.

Proving the PPC property for deterministic sequences is difficult; thus, there are few results in this direction. For example, EI-Baz, Marklof and Vinogradov \cite{EMV15}  proved that the sequence $(\sqrt{n})$ has PPC. Lutsko, Athanasios and Technau \cite{LAT25} proved that $\alpha n^\theta(0<\theta<\frac{14}{41}\text{ and }\alpha>0)$ has PPC, which is later improved to $0<\theta<\frac{43}{117}$ by Radziwiłł and Shubin \cite{RS24}. Lutsko and Technau \cite{LT25} showed that the slowly growing sequence $\alpha(\log n)^A$ has PPC for all $\alpha>0$ and $A>1$. Recently, Hauke \cite{H25} proved that for $(a_n)_{n\in\N}$ arising from the set of rough numbers with explicit roughness parameters and any badly approximable $\alpha$, 
$(a_n\alpha)_{n\in\N}$ has PPC. The metric theory of pair correlation was studied from a mathematical point of view by Rudnick and Sarnak \cite{RS98}. They showed that for almost all $\alpha$, the sequence $(a_n\alpha)_{n\in\N}$ has PPC, where $(a_n)$ is a polynomial of degree at least two. See \cite{AB21,ABTY23,AEM21,ALL17,AY25,RSZ01,RZ99} for more metric PPC results.

We say that the sequence $(x_n)_{n\in\N}$ has weak PPC with parameter $0<\delta<1$ if for all $s>0,$
\[\lim_{N\to\infty}\frac{\sharp \left\{ 1 \leq m \neq n \leq N\colon \|x_m - x_n\| \leq \frac{s}{N^{\delta}} \right\}}{N^{2-\delta}} =2s.\]
To our knowledge, weak PPC was first introduced by Nair and Pollicot \cite{NP03}. 
Steinerberger \cite{Ste17, Ste20} showed it implies equidistribution and viewed it 
as an intermediate property between equidistribution and PPC. Hauke and Zafeiropoulos 
\cite{HZ24} proved that PPC implies weak PPC, but the converse fails. 
Wei{\ss} and Skill \cite{WS21} showed that van der Corput sequences and 
\(\bigl((\sqrt{5}-1)n/2\bigr)_{n\in\mathbb N}\) have weak PPC but not PPC. 
Weiss \cite{Wei24} recently proved that \((n\alpha)_{n\in\mathbb N}\) satisfies 
weak PPC for every badly approximable \(\alpha\), and this was later extended 
to all irrational \(\alpha\) \cite{W22}.

In conclusion, the studies reviewed above yield the relational diagram illustrated in Figure~\ref{homogeneous}.
\begin{figure}[H]
\centering
\begin{tikzpicture}[
    >=Latex,
    every node/.style={font=\normalsize},
    rel/.style={double, double distance=2pt, ->, line width=0.6pt}
]

%------------------------------------------------
% nodes
%------------------------------------------------
\node (T) at (0,4.4) {PPC};
\node (L) at (-6.0,0) {Weak PPC};
\node (R) at (6.0,0) {Equidistribution};

%------------------------------------------------
% left outer: weak PPC -> PPC (blocked by /)
%------------------------------------------------
\draw[rel] (-5.05,0.62) -- (-0.65,3.78)
    node[midway, font=\normalsize] {$\textcolor{red}{/}$};

%------------------------------------------------
% left inner: PPC -> weak PPC
%------------------------------------------------
\draw[rel] (-0.25,3.30) -- (-4.70,0.22);

%------------------------------------------------
% right outer: u.d. mod 1 -> PPC (blocked by /)
%------------------------------------------------
\draw[rel] (5.05,0.62) -- (0.65,3.78)
    node[midway, font=\normalsize] {$\textcolor{red}{/}$};

%------------------------------------------------
% right inner: PPC -> u.d. mod 1
%------------------------------------------------
\draw[rel] (0.25,3.30) -- (4.70,0.22);

%------------------------------------------------
% bottom upper: weak PPC -> u.d. mod 1
%------------------------------------------------
\draw[rel] (-4.10,0.18) -- (4.10,0.18);

%------------------------------------------------
% bottom lower: u.d. mod 1 -> weak PPC (blocked by /)
%------------------------------------------------
\draw[rel] (4.10,-0.32) -- (-4.10,-0.32)
    node[midway, font=\normalsize] {$\textcolor{red}{/}$};

\end{tikzpicture}
\caption{Relations among equidistribution, PPC and weak PPC}
\label{homogeneous}
\end{figure}
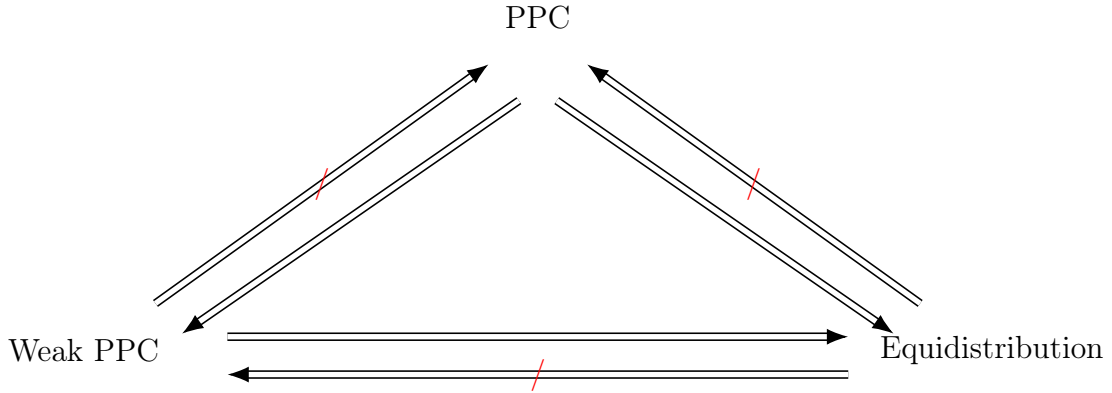%\begin{remark}

\noindent\textbf{Question$\colon$} Are the relations in Figure \ref{homogeneous} valid in the inhomogeneous setting? 
\smallskip

The present paper offers a partial contribution towards this question, falling short of a full resolution. The relations in Figure \ref{homogeneous} are generally expected to persist in the inhomogeneous setting. However,  our results, together with known facts, reveal that distinguishes between PPC and inhomogeneous PPC.

\subsection{Inhomogeneous PPC and its weak form}
For $\gamma \in (0,1)$, a sequence $(x_n)_{n \in \mathbb{N}}\subseteq[0,1)$ is said to have inhomogeneous PPC (also referred to as $\gamma$-PPC) if for every $s > 0$,
\begin{equation}\label{IPPC}
\lim_{N \to \infty}\frac{\sharp \left\{ 1 \leq m \neq n \leq N\colon \|x_m - x_n-\gamma\| \leq \frac{s}{N} \right\}}{N} = 2s.
\end{equation}
A first treatment of the notion of $\gamma$-PPC can be found in \cite{Ram20}.  Hauke and Zafeiropoulos \cite[Theorems 1 and 4]{HZ25} proved that $\gamma$-PPC and equidistribution are independent of each other (see Figure~\ref{inhomogeneous}).

Similarly, for $\gamma\in (0,1)$ and $0<\delta<1,$ we say that the sequence $(x_n)_{n\in\N}$ has weak inhomogeneous PPC (also called weak $\gamma$-PPC)  if for every $s>0,$
\begin{equation}\label{weak}
\lim_{N\to\infty}\frac{\sharp \left\{ 1 \leq m \neq n \leq N\colon \|x_m - x_n-\gamma\| \leq \frac{s}{N^{\delta}} \right\}}{N^{2-\delta}} =2s.
\end{equation}
Since weak $\gamma$-PPC is equivalent to weak $(1-\gamma)$-PPC, it suffices to restrict attention to $0 < \gamma \leq 1/2$.

\begin{theorem}\label{thm1}
For each $\gamma \in (0,\frac{1}{2}]$, and $0<\delta<1$, there is a sequence 
$(x_n)_{n\in\mathbb{N}}$ satisfying weak $\gamma$-PPC but not $\gamma$-PPC.
\end{theorem}

\begin{theorem}\label{thm2}
For each $\gamma \in (0,\frac{1}{2}]$ and $0<\delta<1$, there exists a sequence $(x_n)_{n\in\mathbb{N}}$ with weak $\gamma$-PPC that is not equidistributed.
\end{theorem}

Combining Theorems \ref{thm1} and \ref{thm2} with the results of \cite[Theorems 1 and 4]{HZ25} yields the relations among equidistribution, $\gamma$-PPC, and weak $\gamma$-PPC depicted in Figure~\ref{inhomogeneous}.
It remains open whether $\gamma$-PPC is stronger than weak $\gamma$-PPC; we conjecture that the two notions are in fact incomparable in strength. Also open is the existence of an equidistributed sequence that fails to be weak $\gamma$-PPC.

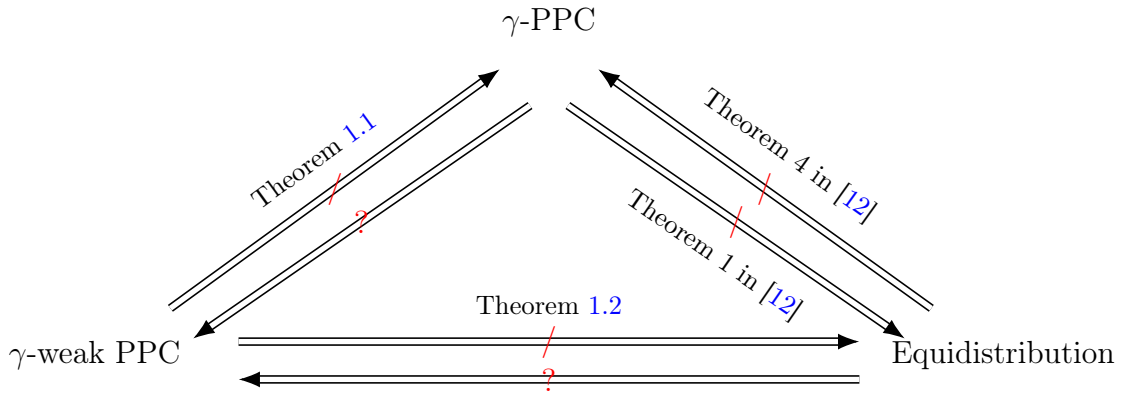
\begin{figure}[H]
\centering
\begin{tikzpicture}[
    >=Latex,
    every node/.style={font=\normalsize},
    rel/.style={double, double distance=2pt, ->, line width=0.6pt}
]

%------------------------------------------------
% nodes
%------------------------------------------------
\node (T) at (0,4.4) {$\gamma$-PPC};
\node (L) at (-6.0,0) {$\gamma$-weak PPC};
\node (R) at (6.0,0) {Equidistribution};

%------------------------------------------------
% left outer: gamma-weak PPC -> gamma-PPC (blocked by /)
%------------------------------------------------
\draw[rel] (-5.00,0.62) -- (-0.65,3.78)
    node[midway, font=\normalsize] {$\textcolor{red}{/}$}
    node[midway, sloped, above=6pt, font=\footnotesize] {Theorem \ref{thm1}};

%------------------------------------------------
% left inner: gamma-PPC -> gamma-weak PPC
%------------------------------------------------
\draw[rel] (-0.25,3.30) -- (-4.70,0.22)
node[midway, font=\normalsize] {$\textcolor{red}{?}$};
    
%------------------------------------------------
% right outer: Equidistribution -> gamma-PPC (blocked by /)
%------------------------------------------------
\draw[rel] (5.05,0.62) -- (0.65,3.78)
node[midway, sloped, above=6pt, font=\footnotesize] {Theorem 4 in \cite{HZ25}}
node[midway, font=\normalsize] {$\textcolor{red}{/}$};

%------------------------------------------------
% right inner: gamma-PPC -> Equidistribution
%------------------------------------------------
\draw[rel] (0.25,3.30) -- (4.70,0.22)
node[midway, font=\normalsize] {$\textcolor{red}{/}$}
    node[midway, sloped, below=6pt, font=\footnotesize] {Theorem 1 in \cite{HZ25}};

%------------------------------------------------
% bottom upper: gamma-weak PPC -> Equidistribution
%------------------------------------------------
\draw[rel] (-4.10,0.18) -- (4.10,0.18)
    node[midway, above=6pt, font=\footnotesize] {Theorem \ref{thm2}}
    node[midway, font=\normalsize] {$\textcolor{red}{/}$};

%------------------------------------------------
% bottom lower: Equidistribution -> gamma-weak PPC (blocked by /)
%------------------------------------------------
\draw[rel] (4.10,-0.32) -- (-4.10,-0.32)
    node[midway, font=\normalsize] {$\textcolor{red}{?}$};

%------------------------------------------------
% additional edge: gamma-PPC -> Equidistribution (with Theorem 1.4 in [13])
% (Note: This is the right inner arrow already labeled above)
% No additional edge needed; the right inner arrow already covers this.
% If you meant a separate edge, we can add it here.

\end{tikzpicture}
\caption{Relations among equidistribution, $\gamma$-PPC and weak $\gamma$-PPC}
\label{inhomogeneous}
\end{figure}

Finally, we turn to the relation between weak $\gamma$-PPC for distinct values of $\gamma$. Theorem \ref{thm2} shows that weak $\gamma$-PPC does not imply weak PPC (interpreted as weak $0$-PPC via \eqref{weak}); otherwise, any sequence satisfying weak $\gamma$-PPC would necessarily be equidistributed. This naturally raises the question of whether an analogous phenomenon holds for weak $\gamma_1$-PPC and weak $\gamma_2$-PPC with distinct $\gamma_1, \gamma_2 \in (0, 1/2]$. The following theorem answers this in the affirmative.

\begin{theorem}\label{thm3}
For any distinct $\gamma_1, \gamma_2 \in (0,\tfrac{1}{2}]$ and $0<\delta<1$, there exists a real sequence $(y_n)_{n\in\mathbb{N}}$ that has weak $\gamma_1$-PPC but not weak $\gamma_2$-PPC.
\end{theorem}

\subsection*{Notation}
We use the standard asymptotic notation: $f\ll g$ and $f=O(g)$ mean that $|f|\le C g$ for some constant $C>0$ and all sufficiently large arguments; $f=o(g)$ means $f/g\to0$; and $f\asymp g$ means both $f\ll g$ and $g\ll f$. If the implied constant depends on a parameter $a$, we write $f\ll_a g$ or $f=O_a(g)$.

\section{Proof of  Theorem \ref{thm1}}
Let $\mathbb{T} = \mathbb{R}/\mathbb{Z}$ be the one-dimensional torus, equipped with the normalized Lebesgue measure $\lambda$ (so that $\lambda(\mathbb{T})=1$). For $\gamma \in (0,\tfrac{1}{2}]$, $0<\delta<1$, $s>0$, $N\in\mathbb{N}$, and $t\in\mathbb{T}$, we define the weak $\gamma$-pair correlation function
\begin{equation}\label{pair}
R_N(\gamma,s,\delta):=\frac{1}{N^{2-\delta}}
\sharp\left\{ 1 \le m \neq n \le N\colon\|x_m - x_n - \gamma\| \le \frac{s}{N^{\delta}} \right\},
\end{equation}
the local counting function
\[
F_N(t,s,\delta):=\frac{1}{N^{1-\delta}}
\sharp\left\{ 1 \le n \le N\colon \|x_n - t\| \le \frac{s}{2N^{\delta}} \right\},
\]
and the integral
\[
I_N(\gamma,s,\delta):=\int_{\mathbb{T}} F_N(t,s,\delta) F_N(t-\gamma,s,\delta) \, \mathrm{d}t.
\]

The following lemma provides a useful integral characterisation of weak $\gamma$-PPC. 

\begin{lemma}\label{integral}
Let $(x_n)_{n\in\mathbb N}\subseteq\mathbb{T}$ be a sequence, and let $0<\delta<1$, $s>0$, and $\gamma\in(0,\tfrac12]$ be fixed. Then the following are equivalent:
\smallskip

{\rm{(1)}}  $ \lim\limits_{N\to\infty} R_N(\gamma,s,\delta) = 2s$; \ \ 
{\rm{(2)}}  $ \lim\limits_{N\to\infty} I_N(\gamma,s,\delta) = s^2$.

\end{lemma}

\begin{proof}
The proof follows the same lines as in \cite[Lemma 9]{HZ24}.
\end{proof}

Let $(y_n)_{n\in\mathbb{N}}$ be a sequence of independent and identically distributed (i.i.d.) random variables uniformly distributed on the one-dimensional  torus $\mathbb{T}$, and define $(x_n)_{n\in\mathbb{N}}$ by
\[
x_{2n-1}=y_n,\qquad x_{2n}=y_n+\gamma.
\]
Such a sequence $(x_n)$ recently also appeared in \cite[Theorem 4]{HZ25}.
For this construction, we observe that
\[
\frac{1}{2N} \sharp \left\{1\le n \neq m \le 2N : \|x_n - x_m - \gamma\| \le \frac{s}{2N} \right\} \ge \frac{1}{2N} N = \frac{1}{2}.
\]
Taking $s < \frac{1}{4}$, it follows from \eqref{IPPC} that $(x_n)$ does not possess the  $\gamma$-PPC property. Next, we show that $(x_n)_{n\in\N}$ possesses the weak $\gamma$-PPC almost surely for any $\gamma \in (0, \tfrac{1}{2}]$, which, together with the previous observation, completes the proof of Theorem \ref{thm1}.

Let $B_N(t,s,\delta) = \{ x \in \mathbb T\colon \|x - t\| \le s/(2N^\delta) \}$, abbreviated as $B_N(t)$ when no confusion can occur.   Then
\[F_N(t, s, \delta) = \frac{1}{N^{1-\delta}} \sum_{n=1}^N \mathbf{1}_{B_N(t)}(x_n),\] 
where $\mathbf{1}_{B_N(t)}$ denotes the indicator function of $B_N(t).$ Since each $x_n$ is uniformly distributed on $\mathbb{T}$, we have
$$\mathbb{P}\big(x_n \in B_N(t)\big) = \lambda(B_N(t)) = \frac{s}{N^\delta}.$$ 
Therefore, the expectation is
\[
\mathbb{E}F_N(t,s,\delta)
= \frac{1}{N^{1-\delta}} \sum_{n=1}^N \frac{s}{N^\delta}
= s.
\]

 Fix $\varepsilon \in (0,2s)$, set $\beta = \frac{\varepsilon}{8s}$, $\alpha = \frac{3\varepsilon}{4}$. For sufficiently large $N$, let $G_N = \{0, \frac{1}{M_N}, \dots, \frac{M_N-1}{M_N}\}$ with $M_N = \left\lceil \frac{N^\delta}{\beta s} \right\rceil$ be a partition of $\mathbb{T}$. Let $\mathcal{P}_N$ denote the family of all intervals in $\mathbb{T}$ with endpoints in $G_N$ and length at most $\frac{2s}{N^{\delta}}$. Then $\sharp \mathcal{P}_N =O_{\varepsilon,s}(N^{2\delta})$.

For $I \in \mathcal{P}_N$, define
\[
Z_N(I) := \sum_{n=1}^N \mathbf{1}_I(x_n).
\]

\begin{lemma}\label{discrepancy}
For all $N$ sufficiently large, almost surely,
\[
|Z_N(I) - N\lambda(I)| < \alpha N^{1-\delta}
\]
for every $I \in \mathcal{P}_N$.
\end{lemma}

\begin{proof}
Fix $I \in \mathcal{P}_N$. Since $\lambda(I) \le 2s/N^\delta \to 0$ and $\gamma > 0$ is fixed, for all sufficiently large $N$ we have
\[
I \cap (I-\gamma) = \emptyset.
\]
Indeed, otherwise there would exist $x \in I$ with $x+\gamma \in I$, implying $\lambda(I) \ge \gamma/2$ for sufficiently small intervals, contradicting $\lambda(I) \le 2s/N^\delta$ for large $N$.

Let $M = \lfloor N/2 \rfloor$ and write $N = 2M + r$, where $r \in \{0,1\}$. For each $n=1,\dots,M$, set
\[
S_n = \mathbf{1}_I(y_n) + \mathbf{1}_I(y_n+\gamma).
\]
Since $I \cap (I-\gamma) = \varnothing$, the events $\{y_n \in I\}$ and $\{y_n \in I-\gamma\}$ are mutually exclusive. Hence
\[
S_n \sim \mathrm{Bernoulli}\big(2\lambda(I)\big),
\]
and the $S_n$'s are independent. Moreover,
\[
Z_N(I) = \sum_{n=1}^M S_n + \mathcal{R}_N,
\]
where
\[
\mathcal{R}_N =
\begin{cases}
0, & r=0,\\
\mathbf{1}_I(y_{M+1}), & r=1,
\end{cases}
\]
with $\mathcal{R}_N$ independent of the $S_n$'s.

Define
\[
T_n =
\begin{cases}
S_n - 2\lambda(I), & 1 \le n \le M,\\
\mathcal{R}_N - r\lambda(I), & n = M+1 \text{ if } r=1.
\end{cases}
\]
Then $T_1,\dots,T_{M+r}$ are independent, mean-zero random variables, with $|T_n| \le 2$ for all $n$, and
\[
\sum_{n=1}^{M+r} \mathbb{E}T_n^2
= M \cdot 2\lambda(I)(1-2\lambda(I)) + r \cdot \lambda(I)(1-\lambda(I))
\le 2M\lambda(I) + r\lambda(I)
= N\lambda(I).
\]
Since $\sum_{n=1}^{M+r} T_n = Z_N(I) - N\lambda(I)$, Bernstein's inequality\footnote{\textbf{Bernstein's inequality:}
Let $X_1, \ldots, X_N$ be independent mean-zero random variables such that $|X_n| \le K$ for all $n$. Then, for every $a \ge 0$,
\[
\mathbb{P} \left\{ \left| \sum_{n=1}^N X_n \right| \ge a \right\}
\le 2 \exp \left( -\frac{a^2 / 2}{\sigma^2 + K a / 3} \right),
\]
where $\sigma^2 = \sum_{n=1}^N \mathbb{E} X_n^2$ is the variance of the sum.}  \cite[Theorem 2.8.4]{VershyninHDP} yields
\[
\mathbb{P}\left( |Z_N(I) - N\lambda(I)| \ge \alpha N^{1-\delta} \right)
\le 2\exp\left(
-\frac{\alpha^2 N^{2-2\delta}/2}
{N\lambda(I) + \frac{2}{3}\alpha N^{1-\delta}}
\right).
\]
Using $\lambda(I) \le 2s/N^\delta$, we have
\[
N\lambda(I) + \frac{2}{3}\alpha N^{1-\delta}
\le \left(2s + \frac{2}{3}\alpha\right) N^{1-\delta},
\]
and therefore
\[
\mathbb{P}\left( |Z_N(I) - N\lambda(I)| \ge \alpha N^{1-\delta} \right)
\le 2\exp\left(
-\frac{\alpha^2}{4s + \frac{4}{3}\alpha} N^{1-\delta}
\right).
\]

Define
\[
E_N = \left\{ I \in \mathcal{P}_N : |Z_N(I) - N\lambda(I)| \ge \alpha N^{1-\delta} \right\}.
\]
Since $\delta \in (0,1)$ and $|\mathcal{P}_N| = O(N^{2\delta})$, we have
\begin{align*}
\sum_{N=1}^{\infty}\mathbb{P}(E_N)
&\le \sum_{N=1}^{\infty}\sum_{I \in \mathcal{P}_N}
\mathbb{P}\left( |Z_N(I) - N\lambda(I)| \ge \alpha N^{1-\delta} \right) \\
&\le 2\sum_{N=1}^{\infty} N^{2\delta}
\exp\left( -\frac{\alpha^2}{4s + \frac{4}{3}\alpha} N^{1-\delta} \right)
< \infty.
\end{align*}
The conclusion follows from the first Borel--Cantelli lemma.
\end{proof}

 \begin{lemma}\label{uniform}
For $t\in \mathbb{T}$, we have almost surely
\begin{equation*}
\lim_{N\to\infty}\sup_{t\in \mathbb{T}} \left| F_{N}(t,s,\delta) - s \right| = 0.
\end{equation*}
\end{lemma}

\begin{proof}
For each $t\in \mathbb{T}$, by the definition of $B_N(t)$, there exist intervals $I_N^-(t), I_N^+(t) \in \mathcal{P}_N$ such that
\[
I_N^-(t) \subseteq B_N(t) \subseteq I_N^+(t),
\]
with
\[
\lambda(I_N^-(t)) \ge \frac{s}{N^\delta} - \frac{2}{M_N}, \qquad 
\lambda(I_N^+(t)) \le \frac{s}{N^\delta} + \frac{2}{M_N}.
\]
Hence,
\begin{align*}
\left|\sum_{n=1}^N \mathbf{1}_{B_N(t)}(x_n) - sN^{1-\delta}\right|
\le 
\left|\sum_{n=1}^N \mathbf{1}_{I_N^-(t)}(x_n) - sN^{1-\delta}\right| + \left|\sum_{n=1}^N \mathbf{1}_{I_N^+(t)}(x_n) - sN^{1-\delta}\right|.
\end{align*}

By Lemma \ref{discrepancy}, for sufficiently large $N$, almost surely,
\begin{align*}
\left|\sum_{n=1}^N \mathbf{1}_{I_N^-(t)}(x_n) - sN^{1-\delta}\right|
&\le 
\left|Z_N(I_N^-(t)) - N\lambda(I_N^-(t))\right| 
+ \left|N\lambda(I_N^-(t)) - sN^{1-\delta}\right| \\
&\le \alpha N^{1-\delta} + 2s\beta N^{1-\delta}
= \varepsilon N^{1-\delta}.
\end{align*}
Similarly, the same argument yields
\[
\left|\sum_{n=1}^N \mathbf{1}_{I_N^+(t)}(x_n) - sN^{1-\delta}\right| \le \varepsilon N^{1-\delta}.
\]
Therefore,
\[
\left|\sum_{n=1}^N \mathbf{1}_{B_N(t)}(x_n) - sN^{1-\delta}\right| \le 2\varepsilon N^{1-\delta}.
\]
Combining the above inequality with the arbitrariness of 
$\varepsilon$ completes the proof.
\end{proof}

\begin{lemma}\label{convergence}
For $\gamma \in (0, \frac{1}{2}]$ and $\delta\in(0,1),$ we have 
\[
\lim_{N \to \infty} I_N(\gamma, s, \delta) = s^2 \quad \text{almost surely}.
\]
\end{lemma}

\begin{proof}
By Lemma \ref{uniform}, we deduce that
\begin{align*}
&\left| \int_{\mathbb{T}} F_{N}(t,s,\delta) F_{N}(t-\gamma,s,\delta) \, \mathrm{d}t - s^2 \right|\\ \leq &\int_{\mathbb{T}} \left| F_{N}(t,s,N) F_{\delta}(t-\gamma,s,\delta) - s^2 \right| \mathrm{d}t\\ \ll &\sup_{t \in \mathbb{T}} \left| F_{N}(t,s,\delta) F_{N}(t-\gamma,s,\delta) - s^2 \right| \to 0 \ \ \text{almost surely.}
\end{align*}
\end{proof}

To sum up, combining Lemmas \ref{integral} and \ref{convergence}, the sequence $(x_n)$ has weak $\gamma$-PPC almost surely, which completes the proof of Theorem \ref{thm1}.

\section{Proof of Theorem \ref{thm2}}

The proof of Theorem \ref{thm2} relies on the following probabilistic statement, whose proof is inspired by \cite[Theorem 2]{ALP18} and \cite[Theorem 2]{HZ25}. We first recall some standard terminology.

We say that a random variable $X$ on a probability space $(\Omega, \Sigma, \mathbb{P})$ has distribution function $G$, or is uniformly distributed with respect to $G$, if
\[
\mathbb{P}(X < x) = G(x), \qquad \text{for all } x \in [0,1].
\]

A function $G \colon [0,1] \to \mathbb{R}$ is called an asymptotic distribution function of a real sequence $(x_n)_{n \ge 1} \subseteq [0,1]$ if
\begin{equation*}
G(x) = \lim_{N \to \infty} \frac{1}{N} \sharp\{ n \leq N\colon 0 \leq x_n \leq x \}, \qquad 0 \le x \le 1.
\end{equation*}
Equivalently, $G$ describes the limiting proportion of terms of the sequence lying in $[0,x]$.

A distribution function $G$ is said to be absolutely continuous if there exists a non-negative integrable function $g \in L^1([0,1])$, called the density of $G$, such that
\[
G(x) = \int_0^x g(t) {\rm{d}}t, \qquad \text{for all } x \in [0,1].
\]
In this case $g = G'$ almost everywhere. In particular, if the sequence $(x_n)_{n\in\N}$ is uniformly distributed, then $G(x) = x$ for all $x \in [0,1]$, and consequently $g(x) \equiv 1$ almost everywhere. Conversely, any density $g \not\equiv 1$ corresponds to a non-uniformly distributed sequence. We shall also require the stronger integrability condition $g \in L^2([0,1])$, which ensures the boundedness and continuity of the autocorrelation integral $\int_0^1 g(x)g(x+\gamma){\rm{d}}x$.\footnote{
Throughout the remainder of the paper, whenever $g\colon [0,1] \to \mathbb{R}$ is a density function and $\gamma \neq 0$, we tacitly extend $g$ to $\mathbb{R}$ periodically with period $1$, and write $\int_0^1 g(x)g(x+\gamma){\rm{d}}x$ in place of the more accurate $\int_0^1 g(x)g(\{x+\gamma\}){\rm{d}}x$.}

\begin{lemma}\label{lemma for theorem 1.2}
Let $\gamma \in (0, \frac{1}{2}]$, $\delta \in (0,1)$, and let $G \colon [0,1] \to \mathbb{R}$ be an absolutely continuous distribution function with density $g \in L^2([0,1])$. Let $(x_n)_{n \ge 1}$ be a sequence of independent random variables on a probability space $(\Omega, \Sigma, \mathbb{P})$, each distributed according to $G$. Then, $\mathbb{P}$-almost surely,
\begin{equation*}
\lim_{N \to \infty} R_N(\gamma, s, \delta) = 2s \cdot \int_0^1 g(x)g(x + \gamma) {\rm{d}}x \quad \text{for all } s>0,
\end{equation*}
where $R_N(\gamma, s, \delta)$ is defined as (\ref{pair}).
\end{lemma}

\begin{proof}

For $n \neq m$, the difference $x_n-x_m$ has probability density function
\begin{equation*}\label{density function of Xi-Xj}
f(t) = \int_0^1 g(x) g(x + t) {\rm{d}}x.
\end{equation*}
For $x \in \mathbb{R}$ and $r > 0$, we write
$B(x, r) := \{ t \in \mathbb{R} : \|t - x\| \le r \},
$ then the random variable 
\begin{align*}
R_N(\gamma, s, \delta)= \frac{1}{N^{2-\delta}} \sum_{\substack{1\le m \neq n \leq N}} \mathbf{1}_{B\left(\gamma, \frac{s}{N^\delta}\right)}(x_m - x_n)
\end{align*}
has expectation 
\begin{equation}\label{epc}
\begin{split}
\mathbb{E}\left[R_N(\gamma, s, \delta)\right]
&= \frac{1}{N^{2-\delta}} \sum_{\substack{1\le m \neq n \leq N}} \int \mathbf{1}_{B\left(\gamma, \frac{s}{N^\delta}\right)}(x_m - x_n) \, \mathrm{d}\mathbb{P} \\
&= \frac{1}{N^{2-\delta}} \sum_{\substack{1\le m \neq n \leq N}} \int_{B\left(\gamma, \frac{s}{N^\delta}\right)} f(t) \, \mathrm{d}t \\
&= \frac{N-1}{N^{1-\delta}} \int_{B\left(\gamma, \frac{s}{N^\delta}\right)} f(t) \, \mathrm{d}t.
\end{split}
\end{equation}
Taking the limit yields that
\begin{align*}   
\lim_{N\to\infty} \mathbb{E}\left[R_N(\gamma, s, \delta)\right]&=\lim_{N\to\infty} \frac{N-1}{N^{1-\delta}} \int_{B\left(\gamma, \frac{s}{N^\delta}\right)} f(t) \, \mathrm{d}t\\
&=\lim_{N\to\infty} N^{\delta} \int_{B\left(\gamma, \frac{s}{N^\delta}\right)} f(t) \, \mathrm{d}t\\
&= \lim_{N\to\infty} 2s \cdot \frac{1}{2s/N^{\delta}} \int_{B\left(\gamma, \frac{s}{N^\delta}\right)} f(t) \, \mathrm{d}t\\
&= 2s \cdot f(\gamma)= 2s \cdot \int_0^1 g(x)g(x + \gamma) \, \mathrm{d}x,
\end{align*}
where the penultimate equality holds since \( f = g * g \) is continuous (as \( g \in L^2 \)), so the integral mean value over the shrinking interval \( B(\gamma, s/N^\delta) \) converges to \( f(\gamma) \).

Following the case distinction in the proof of \cite[Theorem 1.2]{HZ25} (pp.~6--7), a direct computation yields the variance estimate of $R_{N}(\gamma,s,\delta)\colon$
\begin{align*}
\int \left| R_N(\gamma,s,\delta) - \mathbb{E}\left[R_N(\gamma, s,\delta)\right] \right|^2 \, \mathrm{d}\mathbb{P} 
\ll \frac{1}{N^{2-\delta}} \mathbb{E}\left[R_N(\gamma, s, \delta)\right]
\ll \frac{1}{N^{2-\delta}}
\end{align*}
for $N$  large enough, where the final inequality follows by (\ref{epc}) and the continuity of $f$.

Chebyshev's inequality and the variance estimate imply that, for $N_m=m^2$ and any $\varepsilon>0$,
\[
\sum_{m=1}^{\infty} \mathbb{P}\left( \left| R_{N_m}(\gamma,s,\delta) - \mathbb{E}\left[R_{N_m}(\gamma,s,\delta) \right] \right| > \varepsilon \right)
\ll \frac{1}{\varepsilon^2} \sum_{m=1}^{\infty} m^{2\delta-4} 
< \infty,
\]
as $0<\delta<1$. The first Borel--Cantelli lemma then implies that, almost surely,
\[
\lim_{m \to \infty} R_{N_m}(\gamma,s,\delta) 
= \lim_{m \to \infty} \mathbb{E}\left[ R_{N_m}(\gamma,s,\delta) \right]
= 2s \int_0^1 g(x)g(x+\gamma) {\rm{d}}x.
\]
A standard approximation argument (see, e.g., \cite[p.~475]{ALL17}) extends this convergence from the subsequence $(N_m)$ to the full sequence $N \to \infty$, thereby establishing the desired almost-sure limit for each $s > 0$.
This completes the poof.

\end{proof}

In view of Lemma~\ref{lemma for theorem 1.2}, the proof of Theorem~\ref{thm2} reduces to constructing a density function $g$ supported on $[0,1]$ satisfying
\begin{equation*}
\int_0^1 g(x)g(x+\gamma) {\rm{d}}x = 1, \qquad \int_0^1 g(x) {\rm{d}}x = 1, \qquad g(x) \not\equiv 1.
\end{equation*}

We now provide such a construction, distinguishing between the cases $\gamma \in (0, \frac{1}{2})$ and $\gamma = \frac{1}{2}$.
\smallskip

\noindent\underline{\textbf{Case 1:} $0 < \gamma < \frac{1}{2}$}. Choose $0 < \alpha < \frac{1}{4}$ with $\alpha < \min\{\gamma, 1-2\gamma\}$, and let $t_1, t_2$ denote the two (necessarily positive) roots of
\[
\alpha t^2 - t + 1 = 0.
\]
Define
\[
g(x) =
\begin{cases}
t_1, & x \in [0, \alpha), \\
t_2, & x \in [\gamma, \gamma+\alpha), \\
0, & \text{otherwise}.
\end{cases}
\]
\smallskip

\noindent\underline{\textbf{Case 2:} $\gamma = \frac{1}{2}$.} Choose $0 < \alpha < \frac{1}{4}$, and let $t_3, t_4$ be the roots of
\[
2\alpha t^2 - 2t + 1 = 0.
\]
Define
\[
g(x) =
\begin{cases}
t_3, & x \in [0, \alpha), \\
t_4, & x \in [\tfrac{1}{2}, \tfrac{1}{2}+\alpha), \\
0, & \text{otherwise}.
\end{cases}
\]

In both cases, elementary verification shows that $g$ is non-negative, integrates to $1$, satisfies the autocorrelation condition $\int_0^1 g(x)g(x+\gamma){\rm{d}}x = 1$, and is not identically equal to $1$. Hence the desired density function exists for every $\gamma \in (0, \frac{1}{2}]$, and the proof of Theorem~\ref{thm2} is complete.

\section{Proof of Theorem \ref{thm3}}

Fix \(0<\delta<1\) and distinct \(\gamma_1,\gamma_2\in(0,\tfrac12]\). Let \((x_n)_{n\ge1}\) be a sequence of i.i.d.~uniform random variables on the torus \(\mathbb T\). For each \(N\ge1\), set
\[
K_N := \left\lfloor N^{\frac{\delta}{1-\delta}} \right\rfloor, \qquad 
M_N := 2N K_N.
\]
We now construct an infinite sequence $\mathcal Y=(y_n)_{n\ge1}$, which will be shown to exhibit weak $\gamma_1$-PPC but not weak $\gamma_2$-PPC.

For \(N\ge1\) and \(1\le k\le K_N\), define the block
\[
\mathcal{B}_{N,k}
:= (
\underbrace{x_k,\ldots,x_k}_{N\text{ copies}},
\underbrace{x_k+\gamma_2,\ldots,x_k+\gamma_2}_{N\text{ copies}}
),
\]
so that \(\mathcal{B}_{N,k}\) has length \(2N\), with the first \(N\) entries equal to \(x_k\) and the last \(N\) entries equal to \(x_k+\gamma_2\). At level \(N\), define the concatenation
\[
B_N := (\mathcal{B}_{N,1}, \mathcal{B}_{N,2}, \ldots, \mathcal{B}_{N,K_N}),
\]
which has length \(2N K_N = M_N\).

\medskip
\noindent\textbf{Inductive construction.} 
We construct the infinite sequence \(\mathcal Y=(y_n)_{n\ge1}\) by specifying its prefixes recursively so that
\begin{equation}\label{block}
(y_1,\dots,y_{M_N}) = B_N \qquad \forall N\ge1. \
\end{equation}

For \(N=1\), since \(K_1=\lfloor 1^{\delta/(1-\delta)}\rfloor=1\), we have
\[
(y_1,\dots,y_{M_1}) = B_1 = \mathcal{B}_{1,1} = (x_1,\; x_1+\gamma_2),
\]
which satisfies (\ref{block}) for \(N=1\). (More generally, one may take any \(K_1\ge1\); the construction below is independent of this choice.)

Assume that for some \(N\ge1\), the prefix \((y_1,\dots,y_{M_N})\) has been constructed and equals \(B_N\). We extend it to \((y_1,\dots,y_{M_{N+1}})\) in two steps:

\begin{enumerate}
\item[(i)] \textbf{Upgrading old blocks.} 
For each old index \(1\le k\le K_N\), append one additional copy of \(x_k\) followed by one additional copy of \(x_k+\gamma_2\) immediately after the current block \(\mathcal{B}_{N,k}\). That is, we append the pair
\[
(x_k,\; x_k+\gamma_2).
\]
This transforms each old block \(\mathcal{B}_{N,k}\) into the upgraded block
\[
\mathcal{B}_{N+1,k}
= (
\underbrace{x_k,\ldots,x_k}_{N+1\text{ copies}},
\underbrace{x_k+\gamma_2,\ldots,x_k+\gamma_2}_{N+1\text{ copies}}
),
\]
whose length increases from \(2N\) to \(2N+2\).

After this step, the first \(M_N+2K_N\) terms of \(\mathcal Y\) consist of the \(K_N\) upgraded blocks \(\mathcal{B}_{N+1,1},\dots,\mathcal{B}_{N+1,K_N}\).

\item[(ii)] \textbf{Appending new full blocks.}
For each new index \(K_N+1\le k\le K_{N+1}\), append the complete block
\[
\mathcal{B}_{N+1,k}
= (
\underbrace{x_k,\ldots,x_k}_{N+1\text{ copies}},
\underbrace{x_k+\gamma_2,\ldots,x_k+\gamma_2}_{N+1\text{ copies}}
).
\]
\end{enumerate}

The total number of appended terms in both steps is
\[
2K_N + (2N+2)(K_{N+1}-K_N)
= 2(N+1)K_{N+1} - 2N K_N
= M_{N+1} - M_N,
\]
so the extended prefix has length \(M_{N+1}\). By construction, it equals
\[
B_{N+1}
= (\mathcal{B}_{N+1,1}, \mathcal{B}_{N+1,2}, \ldots, \mathcal{B}_{N+1,K_N}, \mathcal{B}_{N+1,K_N+1}, \ldots, \mathcal{B}_{N+1,K_{N+1}}).
\]
Hence (\ref{block}) holds for \(N+1\). This completes the induction.

\medskip
\noindent\textbf{Well-definedness.} 
Since the inductive step specifies exactly \(M_{N+1}-M_N\) new terms after the first \(M_N\) terms, and these new terms depend only on \(K_N,K_{N+1}\) and the i.i.d.~sequence \((x_k)\), the infinite sequence \(\mathcal Y=(y_n)_{n\ge1}\) is well-defined. Moreover, by (\ref{block}), for every \(N\ge1\) the first \(M_N\) terms of \(\mathcal Y\) coincide with \(B_N\); in particular, each \(x_k\) and \(x_k+\gamma_2\) appears exactly \(N\) times among the first \(M_N\) terms, for every \(1\le k\le K_N\).

\medskip

In the rest of the paper, we write $R_{N}^{\mathcal{Y}}(\gamma,s,\delta)$ for the weak $\gamma$-PPC pair correlation functions of $\mathcal{Y}.$

\medskip
\noindent\textbf{Step 1: $\mathcal{Y}$ does not have weak $\gamma_2$-PPC.}

We consider the subsequence $N_r := M_r =2 r\lfloor r^{\delta/(1-\delta)} \rfloor$. 
Within the block $B_r$, every sub-block contains $r$ copies of $x_k$ and $r$ copies of $x_k+\gamma_2$. 
For each such sub-block, every pair consisting of one element from each group differs by exactly $\gamma_2$, 
and hence contributes to the $\gamma_2$-counting function at scale $s>0$, 
provided $s/N_r^\delta < 1$ (which holds for all sufficiently large $r$). 

Since all contributions to the counting function are non-negative, we have the lower bound
\begin{align*}
R_{N_{r}}^{\mathcal{Y}}(\gamma_2,s,\delta)
&= \frac{1}{N_r^{2-\delta}}
\sum_{1\le i\neq j \le N_r}
\mathbf{1}_{B(0,\frac{s}{N_{r}^{\delta}})}(y_i-y_j-\gamma_2) \\
&\ge \frac{1}{N_r^{2-\delta}} \cdot K_r r^{2},
\end{align*}
where the inequality follows by restricting the summation to ordered pairs $(i,j)$ lying in the same sub-block, with $y_i$ chosen from the $r$ copies of $x_k+\gamma_2$ and $y_j$ chosen from the $r$ copies of $x_k$; for each such pair we have $\|y_i-y_j-\gamma_2\|=0$, and there are exactly $K_r$ such pairs.

Now $K_r = \lfloor r^{\delta/(1-\delta)} \rfloor$, and $N_r = 2rK_r \asymp 2r^{1/(1-\delta)}$. Hence
\[
K_r r^2 = r^{\frac{2-\delta}{1-\delta}},
\]
while
\[
N_r^{2-\delta} \asymp \left(2r^{\frac{1}{1-\delta}}\right)^{2-\delta}
=2^{2-\delta} r^{\frac{2-\delta}{1-\delta}}.
\]
Therefore the quotient is bounded below by a positive constant:
\[
R_{N_{r}}^{\mathcal{Y}}(\gamma_2,s,\delta) \ge c > 0
\]
for all sufficiently large $r$.

Taking $s \le c/4$, we have $2s < c$. Thus along the subsequence $N_r$,
\[
\liminf_{r\to\infty} R_{N_r}^{\mathcal{Y}}(\gamma_2,s,\delta) \ge c > 2s,
\]
which precludes convergence to $2s$. Consequently, $\mathcal{Y}$ does not have weak $\gamma_2$-PPC.

\medskip
\noindent\textbf{Step 2: $\mathcal{Y}$ has weak $\gamma_1$-PPC almost surely.}

We now prove that $\mathcal{Y}$ has weak $\gamma_1$-PPC. By Lemma~\ref{lemma for theorem 1.2}, 
$(x_n)$ has $\gamma$-weak PPC almost surely for every non-integer $\gamma$; 
in particular, for $\gamma=\gamma_1$, $\gamma=\gamma_1-\gamma_2$, and $\gamma=\gamma_1+\gamma_2$.

Fix $s>0$. We evaluate $R_{N_r}^{\mathcal{Y}}(\gamma_1,s,\delta)$ at the subsequence $N_r=M_r$. 
Let $1\le i\neq j \le N_r$. Write $i=(k-1)r+p$ and $j=(\ell-1)r+q$, where $k,\ell\in\{1,\dots,K_r\}$ indicate the sub-block, and $p,q\in\{1,\dots,r\}$ indicate the position within the sub-block. We distinguish two cases according to whether the indices lie in the same sub-block or in different sub-blocks.

\smallskip
\noindent\textbf{Case 2.1: Same sub-block ($k=\ell$).} 
Within a single sub-block, the elements are either $r$ copies of $x_k$ or $r$ copies of $x_k+\gamma_2$. The difference $y_i-y_j$ can therefore take only three values: $0$, $\gamma_2$, or $-\gamma_2$. 
Consequently, the condition $\|y_i-y_j-\gamma_1\|\le s/N_r^\delta$ reduces respectively to
\[
\|\gamma_1\|\le \frac{s}{N_r^\delta},\qquad
\|\gamma_1-\gamma_2\|\le \frac{s}{N_r^\delta},\qquad
\|\gamma_1+\gamma_2\|\le \frac{s}{N_r^\delta}.
\]
Since $\gamma_1,\gamma_1-\gamma_2,\gamma_1+\gamma_2\notin\mathbb Z$ and $s/N_r^\delta\to 0$, these inequalities fail for all sufficiently large $r$. Hence the same-sub-block contribution is identically zero for large $r$.

The number of ordered pairs $(p,q)$ within the same sub-block is:
\begin{itemize}
\item $r^2+r^2-2r$ pairs for which both elements are copies of the same value (either both $x_k$ or both $x_k+\gamma_2$); here we subtract $r$ to exclude the $p=q$ self-pairs, since the definition of PPC requires $i\neq j$;
\item $r^{2}$ pairs for which $y_i=x_k+\gamma_2$ and $y_j=x_k$ (difference $\gamma_2$);
\item $r^{2}$ pairs for which $y_i=x_k$ and $y_j=x_k+\gamma_2$ (difference $-\gamma_2$).
\end{itemize}

\noindent\textbf{Case 2.2: Different sub-blocks ($k\neq \ell$).} 
Here $x_k$ and $x_\ell$ are independent copies from $(x_n)$. For each pair $(p,q)$, the difference $y_i-y_j$ takes one of four possible forms:
\[
x_k-x_\ell,\qquad x_k-(x_\ell+\gamma_2),\qquad (x_k+\gamma_2)-x_\ell,\qquad (x_k+\gamma_2)-(x_\ell+\gamma_2).
\]
The condition $\|y_i-y_j-\gamma_1\|\le s/N_r^\delta$ thus becomes one of the following three inequivalent conditions:
\[
\|x_k-x_\ell-\gamma_1\|\le \frac{s}{N_r^\delta},\qquad
\|x_k-x_\ell-(\gamma_1-\gamma_2)\|\le \frac{s}{N_r^\delta},\qquad
\|x_k-x_\ell-(\gamma_1+\gamma_2)\|\le \frac{s}{N_r^\delta}.
\]
For fixed $k\neq \ell$, the number of ordered pairs $(p,q)$ yielding each condition is:
\begin{itemize}
\item $r^2+r^2$ pairs for the condition with $\gamma_1$ (both $p,q$ in first $r$ positions gives $r^2$ pairs; both in last $r$ positions gives $r^2$ pairs);
\item $r^{2}$ pairs for the condition with $\gamma_1-\gamma_2$ (when $p$ is in the last $r$ positions and $q$ is in the first $r$ positions);
\item $r^{2}$ pairs for the condition with $\gamma_1+\gamma_2$ (when $p$ is in the first $r$ positions and $q$ is in the last $r$ positions).
\end{itemize}

Combining Cases 2.1 and 2.2, and noting that there are $K_r$ sub-blocks for the same-sub-block contribution and $K_r(K_r-1)$ ordered pairs of distinct sub-blocks for the different-sub-block contribution, we obtain the following expansion:
\begin{align*}
R_{N_r}^{\mathcal{Y}}(\gamma_1,s,\delta)
&= \frac{K_r}{N_r^{2-\delta}}
\Bigg[
\underbrace{ (r^2+r^2-2r)\mathbf{1}_{B(0,\frac{s}{N_{r}^{\delta}})}(\gamma_1)}_{\text{same sub-block: both copies of }x_k\text{ or both copies of }x_k+\gamma_2} \\
&\qquad +  \underbrace{r^{2}\mathbf{1}_{B(0,\frac{s}{N_{r}^{\delta}})}(\gamma_1-\gamma_2) }_{\text{same sub-block: }y_i=x_k+\gamma_2,\ y_j=x_k} \\
&\qquad + \underbrace{r^{2}\mathbf{1}_{B(0,\frac{s}{N_{r}^{\delta}})}(\gamma_1+\gamma_2) }_{\text{same sub-block: }y_i=x_k,\ y_j=x_k+\gamma_2}
\Bigg] \\
&\quad + \frac{1}{N_r^{2-\delta}}
\sum_{1\le k\neq \ell\le K_r}
\Bigg[
(r^2+r^2)\, \mathbf{1}_{B(0,\frac{s}{N_{r}^{\delta}})}(x_k-x_\ell-\gamma_1) \\
&\qquad\qquad + r^{2}\, \mathbf{1}_{B(0,\frac{s}{N_{r}^{\delta}})}\big(x_k-x_\ell-(\gamma_1-\gamma_2)\big) \\
&\qquad\qquad + r^{2}\, \mathbf{1}_{B(0,\frac{s}{N_{r}^{\delta}})}\big(x_k-x_\ell-(\gamma_1+\gamma_2)\big)
\Bigg].
\end{align*}
Since $\gamma_1,\gamma_1-\gamma_2,\gamma_1+\gamma_2$ are all non-integer and $s/N_r^\delta\to 0$, each of the three same-sub-block indicators vanishes for all sufficiently large $r$. Thus these terms contribute $0$ in the limit.

For the different-sub-block terms, applying Lemma~\ref{lemma for theorem 1.2} to each of the three non-integer shifts $\gamma_1$, $\gamma_1-\gamma_2$, and $\gamma_1+\gamma_2$, we have almost surely:
\[
\sum_{1\le k\neq \ell\le K_r}
\mathbf{1}_{B(0,\frac{s}{N_{r}^{\delta}})}(x_k-x_\ell-\gamma)
=\frac{K_r(K_r-1)}{N_r(N_r-1)}N_{r}^{2-\delta}(2s+o(1)).
\]
Substituting this into the expansion yields
\begin{align*}
R_{N_r}^{\mathcal{Y}}(\gamma_1,s,\delta)
&\asymp  \frac{1}{N_r^{2-\delta}}
\Bigg[
4r^{2}\cdot\frac{K_r^{2}}{N_r^{2}}N_{r}^{2-\delta}\big(2s+o(1)\big)
\Bigg] \\
&= \frac{K_r^{2}}{ N_r^{2}} (4r)^2\cdot \big(2s+o(1)\big) \\
&= 2s + o(1)
\end{align*}
for sufficiently large $r$. Therefore,
\[
\lim_{r\to\infty} R_{N_r}^{\mathcal{Y}}(\gamma_1,s,\delta)
 = 2s \quad \text{almost surely}.
\]

Since the subsequence $N_r = 2r\lfloor r^{\delta/(1-\delta)} \rfloor$ satisfies 
$N_{r+1}/N_r \to 1$ as $r\to\infty$, a standard approximation argument extends this convergence from the subsequence $(N_r)$ 
to the full sequence $N\to\infty$. Hence
\[
\lim_{N\to\infty} R_{N}^{\mathcal{Y}}(\gamma_1,s,\delta)
 = 2s \quad \text{almost surely},
\]
so $\mathcal{Y}$ has weak $\gamma_1$-PPC almost surely.

Combining Steps 1 and 2, $\mathcal{Y}$ has weak $\gamma_1$-PPC but not weak
$\gamma_2$-PPC for every $0<\delta<1$. This completes the proof.

\subsection*{Acknowledgements}
The authors express their sincere gratitude to Professor Bo Tan (HUST) for his valuable suggestions, which greatly improved the exposition of this paper. We would also like to thank Christian Wei{\ss} for helpful comments and for bringing reference \cite{W22} to our attention.  This work was supported by the Fundamental Research Funds for the Central Universities.

\author{Zhiqin  Tang}
{\footnotesize
\medskip

School of Mathematics and Statistics

 Wuhan University of Technology, 430070 Wuhan, PR China

Email: \texttt{ZhiqinTang@whut.edu.cn}}
\vspace{5mm}

\author{Qing-Long Zhou}
{\footnotesize
\medskip

School  of  Mathematics  and  Statistics

 Wuhan University of Technology, 430070 Wuhan, PR China

Email: \texttt{zhouql@whut.edu.cn}}

\end{document}